\documentclass[12pt]{amsproc}
\newtheorem{theorem}{\sc Theorem}[section]
\newtheorem{lemma}[theorem]{\sc Lemma}
\newtheorem{proposition}[theorem]{\sc Proposition}
\newtheorem{corollary}[theorem]{\sc Corollary}

\begin{document}
\title[Double automorphisms]{Double automorphisms of graded Lie algebras}
\author{Cristina Acciarri}

\address{(Cristina Acciarri)  Department of Mathematics, University of Brasilia,
Brasilia-DF, 70910-900 Brazil}
\email{acciarricristina@yahoo.it}

\author{Pavel Shumyatsky} 

\address{(Pavel Shumyatsky) Department of Mathematics, University of Brasilia,
Brasilia-DF, 70910-900 Brazil}

\email{pavel@unb.br}

\thanks{Supported by CAPES and CNPq-Brazil}

\begin{abstract} We introduce the concept of a double automorphism of an $A$-graded Lie algebra $L$. Roughly, this is an automorphism of $L$ which also induces an automorphism of the group $A$. It is clear that the set of all double automorphisms of $L$ forms a subgroup in $Aut\, L$. In the present paper we prove several nilpotency criteria for a graded Lie algebra admitting a finite group of double automorphisms. One of the obtained results is as follows.

Let $A$ be a torsion-free abelian group and $L$ an $A$-graded Lie algebra in which $[L,\underbrace{L_0,\ldots,L_0}_{k}]=0$. Assume that $L$ admits a finite group of double automorphisms $H$ such that $C_A(h)=0$ for all nontrivial $h\in H$ and $C_L(H)$ is nilpotent of class $c$. Then $L$ is nilpotent and the class of $L$ is bounded in terms of $|H|$, $k$ and $c$ only.

We also give an application of our results to groups admitting a Frobenius group of automorphisms.
 
\end{abstract}

\maketitle

\section{Introduction}

Let $A$ be an additively written abelian group and $L$ an $A$-graded Lie algebra. Throughout this paper the term ``Lie algebra" means Lie algebra over some commutative ring with unity. Thus we have $L=\displaystyle\bigoplus_{a \in A}L_a$, where $L_a$ are subspaces of $L$ such that $[L_a,L_b]\subseteq L_{a+b}$ for all $a,b\in A$. Elements of the grading components $L_a$ are called \emph{homogeneous}. A subspace $K$ of $L$ is called \emph{homogeneous} if $K=\bigoplus_a(K\cap L_a)$. In this case we write $K_a=K\cap L_a$. Clearly, any subalgebra or ideal generated by homogeneous subspaces is homogeneous. A homogeneous subalgebra and the quotient over a homogeneous ideal can be regarded as $A$-graded Lie algebras with the induced gradings. A subspace $K$ is called ad-nilpotent of index (at most) $n$ if $[L,\underbrace{K,\ldots,K}_{n}]=0$.

Suppose that $\phi$ is an automorphism of $L$ with the property that for every $a\in A$ there exists $b\in A$ such that $L_a^\phi=L_b$. In this case we can define the map $\phi: A\rightarrow A$ by the rule $a^\phi=b$. If the defined map is an automorphism of $A$, we say that $\phi$ is a {\it double automorphism} of the $A$-graded Lie algebra $L$. It is clear that a product of two double automorphisms of $L$ and an inverse of a double automorphism are again double automorphisms. Thus, the set of all double automorphisms of $L$ forms a subgroup in $Aut\, L$. If $H$ is a group of double automorphisms of $L$, we can consider the fixed points of $H$ in both $L$ and $A$. Thus, $$C_L(H)=\{l\in L\, \vert\, l^\phi=l \text{ for all } \phi\in H \}$$ and $$C_A(H)=\{a\in A\, \vert\, L_a^\phi=L_a \text{ for all } \phi\in H \}.$$ The double automorphisms of graded Lie algebras have been implicitly considered in \cite{mashu,dan,khumashu,caeufe}. They arise quite naturally in the study of groups acted on by Frobenius groups of automorphisms. To formulate some of the results obtained in those papers let us use the symbol $|B|$ for the order of a finite group $B$ (and $|\phi|$ for the order of  an automorphism $\phi$). The following theorem can be easily deduced from the results obtained in \cite{khumashu}.

\begin{theorem}\label{kms} Let $A$ be a finite cyclic group and $L$ an $A$-graded Lie algebra such that $L_0=0$. Assume that $L$ admits a double automorphism $\phi$ such that $C_A(\phi^i)=0$ for all $i=1,2,\dots$ and $C_L(\phi)$ is nilpotent of class $c$. Then $L$ is nilpotent and the class of $L$ is bounded in terms of $|\phi|$ and $c$ only.
\end{theorem}
We note that under the hypothesis of the above theorem the order of $\phi$ must be finite. We also note that by the well-known Kreknin theorem \cite{kr} $L$ is soluble with $|A|$-bounded derived length. In the particular case where $A$ has prime order Higman's theorem \cite{hi} (see also Kreknin and Kostrikin \cite{kr-ko}) says that $L$ is even nilpotent of $|A|$-bounded nilpotency class. Hence, the most interesting part of the statement of Theorem \ref{kms} is that the class of $L$ is bounded in terms of $|\phi|$ and $c$ only and thus is independent of $|A|$. There are examples in \cite{khumashu} showing that Theorem \ref{kms} fails if $A$ is not cyclic. On the other hand, the following theorem can be deduced from the resuslts obtained in \cite{caeufe}.

\begin{theorem}\label{cef} Let $A$ be an abelian group and $L$ an $A$-graded Lie algebra such that $L_a$ is ad-nilpotent of index $k$ for all $a\in A$. Assume that $L$ admits a finite group of double automorphisms $H$ such that $C_A(h)=0$ for all nontrivial $h\in H$ and $C_L(H)$ is nilpotent of class $c$. Then $L$ is nilpotent and the class of $L$ is bounded in terms of $|H|$, $k$ and $c$ only.
\end{theorem}

In the present paper we will show that in the case where $A$ is torsion-free the assumption in the above theorem that $L_a$ is ad-nilpotent for all $a\in A$ can be replaced by a much weaker hypothesis.

\begin{theorem}\label{tri} Let $A$ be a torsion-free abelian group and $L$ an $A$-graded Lie algebra such that $L_0$ is ad-nilpotent of index $k$. Assume that $L$ admits a finite group of double automorphisms $H$ such that $C_A(h)=0$ for all nontrivial $h\in H$ and $C_L(H)$ is nilpotent of class $c$. Then $L$ is nilpotent and the class of $L$ is bounded in terms of $|H|$, $k$ and $c$ only.
\end{theorem}

If $B$ is a subgroup of $A$, we denote by $L_B$ the sum $\bigoplus_{a \in B}L_a$ which of course is a subalgebra of $L$. Using parts of the proof of Theorem \ref{tri} we will deduce the following theorem.

\begin{theorem}\label{posa} There exists a constant $P=P(c,q)$ depending only on $c$ and $q$ with the following property: Let $A$ be an elementary group of order $p^2$ where $p\geq P$ and let $L$ be an $A$-graded Lie algebra such that $L_B$ is nilpotent of class at most $c$ for every proper subgroup $B<A$. Assume that $L$ admits a group of double automorphisms $H$ such that $C_A(h)=0$ for all nontrivial $h\in H$ and $C_L(H)$ is nilpotent of class at most $c$. Then $L$ is nilpotent and the class of $L$ is bounded in terms of $c$ and $|H|$ only.
\end{theorem}

Theorem \ref{posa} has an immediate applications to groups admitting Frobenius groups of automorphisms. Recall that a Frobenius group $FH$ with kernel $F$ and complement $H$ can be characterized as a finite group that is a semidirect product of a normal subgroup $F$ by $H$ such that $C_F(h)=1$ for every $h\in H\setminus\{1\}$. By Thompson's theorem \cite{th} the kernel $F$ is nilpotent, and by Higman's theorem \cite{hi} the nilpotency class of $F$ is bounded in terms of the least prime divisor of $|H|$ (explicit upper bounds for the nilpotency class are due to Kreknin and Kostrikin \cite{kr,kr-ko}). The following result was obtained in \cite{caeufe}.
\medskip

{\it Let $FH$ be a Frobenius group with abelian kernel $F$ of rank at least three and with complement $H$ of order $q$. Suppose that $FH$ acts coprimely on a finite group $G$ in such a manner that $C_{G}(H)$ and $C_{G}(a)$ are nilpotent of class at most $c$ for all $a\in F \setminus \{1\}$. Then $G$ is nilpotent of $(c,q)$-bounded class.}
\medskip

The assumption that the rank of $F$ is at least three in the above result is essential. In \cite{khumashu} one can find examples of finite $p$-groups $G$ of unbounded derived length that admit an action of the Frobenius group $FH$ of order 12 such that $C_G(H)$ and $C_G(a)$ are abelian for all nontrivial $a\in F$. Yet, from our Theorem \ref{posa} we can easily derive the following result. 

\begin{theorem}\label{grsa} Let $P$ be as in Theorem \ref{posa} and  $p$ a prime bigger than $P$. Let $FH$ be a Frobenius group whose kernel $F$ is non-cyclic of order $p^2$ and complement $H$ is of order $q$. Suppose that $FH$ acts coprimely on a finite group $G$ in such a manner that $C_{G}(H)$ and $C_{G}(a)$ are nilpotent of class at most $c$ for all $a\in F \setminus \{1\}$. Then $G$ is nilpotent of $(c,q)$-bounded class.
\end{theorem}

As we have mentioned earlier the above theorem fails unless a restriction on the prime $p$ is made. On the other hand, in the proof of Theorem \ref{grsa} we establish that if a finite group $G$ admits a coprime action by a Frobenius group $FH$ where $F$ is non-cyclic abelian and $C_{G}(H)$ and $C_{G}(a)$ are nilpotent for all $a\in F\setminus \{1\}$, then $G$ is nilpotent (see Proposition \ref{nilpo}). This holds without any restrictions on prime divisors of $F$.

Throughout the paper we use the expression ``$(m,n)$-bounded'' for ``bounded above in terms of $m,\,n$ only''.

\section{Theorems \ref{cef} and \ref{tri}}

In this section we will describe the proofs of Theorem \ref{cef} and Theorem \ref{tri}. In the case where $A$ is finite Theorem \ref{cef} was essentially proved in \cite{caeufe}. However it is easy to see that the proof given in \cite{caeufe} can be adapted to the case of infinite $A$ with only minor changes. 

Let $A$ be the additively written abelian group and assume that the finite group $H$ acts on $A$ by automorphisms. We also assume that $|H|=q$ and $C_{A}(h)=0$ for every $h\in H\backslash\{1\}$. We will work with finite sequences of non-zero elements of $A$. Let $a_1,\ldots, a_s$ be not necessarily distinct non-zero elements in $A$. We say that the sequence $(a_1,\ldots, a_s)$ is $H$-\emph{dependent} if and only if there exist distinct $i_1,\ldots,i_m\in \{1,2,\ldots,s\}$ and not necessarily distinct $h_1,\ldots,h_m \in H\setminus \{1\}$ such that $$a_{i_{1}}+\cdots+a_{i_{m}}=a_{i_{1}}^{h_1}+\cdots +a_{i_{m}}^{h_m}.$$

If the sequence $(a_1,\ldots, a_s)$ is not $H$-{dependent}, we will call it 
$H$-independent. We denote by $F_s$ the set of all $H$-independent 
sequences of length $s$ and $Z_s$ the set of all $H$-dependent ones.

\begin{proof}[Proof of Theorem \ref{cef}] By \cite[Theorem 3.9]{caeufe} the nilpotency of $L$ will be established once it is shown that
\begin{equation}\label{Cond 1}
[L_{a_{1}},L_{a_{2}},\ldots,L_{a_{c+1}}]=0 \ \textrm{ whenever } (a_1,a_{2},\ldots,a_{c+1})\in F_{c+1}
\end{equation}

\noindent and
\begin{equation}\label{Cond 2}
[L_a,\underbrace{L_b,\ldots,L_b}_{k}]=0 \ \textrm{ for all } a,b\in A.
\end{equation}

In \cite{khumashu} the condition (\ref{Cond 1}) was called ``selective nilpotency of $L$". Since the condition (\ref{Cond 2}) is immediate from the hypothesis, it is sufficient to show that $L$ satisfies the selective nilpotency condition.

Let $h_1,\dots,h_{q-1}$ be the nontrivial elements of $H$. We will use the fact that for any $u\in L$ the sum $$u+{u}^{h_1}+\cdots+{u}^{h_{q-1}}$$ belongs to the nilpotent subalgebra $C_L(H)$. Choose arbitrarily a sequence of non-zero elements $a_1,a_2,\dots,a_{c+1}\in A$ and elements $$x_{a_1}\in L_{a_1},\ \ \ldots,\ \ x_{a_{c+1}}\in L_{a_{c+1}}.$$ Consider the elements
\begin{eqnarray}
\begin{array}{ccl}
  X_1&=&x_{a_{1}}+x_{a_{1}}^{h_1}+\cdots+x_{a_{1}}^{h_{q-1}}, \\
 & \vdots &  \\
  X_{c+1}&=&x_{a_{c+1}}+x_{a_{c+1}}^{h_1}+\cdots+x_{a_{c+1}}^{h_{q-1}}.
\end{array}
\end{eqnarray}

Since all of them lie in $C_{L}(H)$ and since $C_{L}(H)$ is nilpotent 
of class $c$, it follows that $$[X_1,\ldots,X_{c+1}]=0.$$

After expanding the brackets, the left-hand side of the above equality involves the term 
$[x_{a_1},\ldots,x_{a_{c+1}}]$. Suppose that the commutator 
$[x_{a_1},\ldots,x_{a_{c+1}}]$ is non-zero. Then there must be other non-zero 
terms in the expanded expression that belong to the same component 
$L_{a_1+\cdots+a_{c+1}}$. So there exist $i_1,\ldots,i_m\in \{1,2,\ldots,c+1\}$ and $h_1,\ldots,h_m\in H\setminus\{1\}$ such that
$$a_{i_{1}}+\cdots+a_{i_{m}}=a_{i_{1}}^{h_1}+\cdots+a_{i_{m}}^{h_{m}}.$$

Thus, we have shown that $[x_{a_1},\ldots,x_{a_{c+1}}]=0$ whenever the 
sequence $(a_{1},\dots,a_{c+1})$ is $H$-independent. This completes the proof.
\end{proof}

We will now embark on the proof of Theorem \ref{tri}. For a sequence $(a_1,\ldots,a_s)\in F_s$ we denote by $M(a_1,\ldots,a_s)$ the set of all $v\in A$ such that $(a_1,\ldots,a_s,v)\in Z_{s+1}$. The next lemma was proved in \cite{caeufe} under the additional hypothesis that $A$ is finite. Here we make no assumptions on the order of $A$.

\begin{lemma}\label{LemmaM}
If $(a_1,\ldots,a_s)\in F_s$, then $|M(a_1,\ldots,a_s)|\leq q^{s+1}$.
\end{lemma}

\begin{proof}
Suppose that $(a_1,\ldots,a_s,v)\in Z_{s+1}$. We have $$a_1+a_2+\cdots+a_s+v=a_1^{h_1}+a_2^{h_2}+\cdots+a_s^{h_s}+v^{h_{0}}$$ for suitable $h_i\in H$, where not all of the $h_i$ are equal 1. In fact, $h_0\neq1$, for otherwise the sequence $(a_1,a_2,\ldots,a_s)$ would not be $H$-independent. We see that $$v^{(1-h_{0})}=a_1^{h_1}+a_2^{h_2}+\cdots+a_s^{h_s}-a_1-\cdots-a_s.$$ Since $C_{A}(h_0)=0$, it follows that the map that takes $x$ to $x^{(1-h_{0})}$ is injective on $A$. Hence, we write $$v=(a_1^{h_1}+a_2^{h_2}+\cdots+a_s^{h_s}-a_1-\cdots-a_s)^{(1-h_0)^{-1}}.$$ There are at most $q^{s+1}$ possible choices for the automorphisms $h_i$ in the above equality. Therefore, there are at most $q^{s+1}$ possibilities for $v$, as required.
\end{proof}

From now on in this section $A$ will be assumed to be a torsion-free abelian group.
\begin{lemma}\label{free} Let $n$ be a positive integer and $h_1,\dots,h_n$ not necessarily distinct elements of $H\setminus\{1\}$. If $na=a^{h_1}+\dots+a^{h_n}$ for an element $a\in A$, then $a=0$.
\end{lemma}
\begin{proof} Without loss of generality we can assume that $A=\langle a^H\rangle$. Thus, $A$ is a finitely generated torsion-free group. It follows that $A$ is isomorphic with a direct sum of finitely many copies of $\Bbb Z$ and the group of automorphisms of $A$ is isomorphic with $GL_m(\Bbb Z)$ for some $m$. Hence, $A$ naturally embeds into $\underbrace{{\Bbb C}\oplus\dots\oplus{\Bbb C}}_m$ and we also have an induced embedding of $H$ into $GL_m(\Bbb C)$. Since any complex representation of a finite group is equivalent to a unitary one \cite[Exercise 10.6]{cr}, without loss of generality we can assume that the induced embedding of $H$ into $GL_m(\Bbb C)$ is actually an embedding into $U_m(\Bbb C)$. But then the triangle inequality shows that $$na=a^{h_1}+\dots+a^{h_n}$$ is possible only in the case where $a=0$.
\end{proof}
The following corollary is now immediate.
\begin{corollary}\label{bbb} If $0\neq b\in A$, then the sequence $(b,b,\dots,b)$ is $H$-independent.
\end{corollary}
\begin{proof}[Proof of Theorem \ref{tri}] By \cite[Theorem 3.9]{caeufe} it is sufficient to show that $L$ satisfies the $c$th selective nilpotency condition, that is,
$$[L_{a_{1}},L_{a_{2}},\ldots,L_{a_{c+1}}]=0 \ \textrm{ whenever } (a_1,a_{2},\ldots,a_{c+1})\in F_{c+1}$$
and that there is a $(c,k,q)$-bounded number $m$ such that $$[L_a,\underbrace{L_b,\ldots,L_b}_{m}]=0 \ \textrm{ for all } a,b\in A.$$
The selective nilpotency of $L$ can be established precisely as in the proof of Theorem \ref{cef}. Hence, it is sufficient to show that $L$ verifies the other condition. Choose $a,b\in A$. If $b=0$, by the hypothesis we have $[L_a,\underbrace{L_b,\ldots,L_b}_{k}]=0$. Therefore we assume that $b\neq 0$. Put $$D(b)=\{a\in A\vert\ [L_a,\underbrace{L_b,\ldots,L_b}_{c}]\neq0\} \text{ \ \ and \ \ } M=M(\underbrace{b,\dots,b}_c).$$ We have $D(b)\subseteq M\cup \{0\}$ and so by Lemma \ref{LemmaM} $|D(b)|\leq q^{c+1}+1$. Put $a_i=a+ib$, where $i=0,1,\dots$. Since $A$ is torsion-free, $a_i\neq a_j$ whenever $i\neq j$. Let $d$ be the least index for which $a_d\not\in D(b)$. It is clear that $d\leq q^{c+1}+2$. By the choice of  $d$ we have $$[L_{a_d},\underbrace{L_b,\ldots,L_b}_{c}]=0.$$
Taking into account that  $$[L_a,\underbrace{L_b,\ldots,L_b}_{d}]\leq L_{a_d}$$ write $$[L_a,\underbrace{L_b,\ldots,L_b}_{d+c}]\leq [L_{a_d},\underbrace{L_b,\ldots,L_b}_{c}]=0.$$ Since $d\leq q^{c+1}+2$, the theorem follows.
\end{proof}

\section{Proofs of Theorems \ref{posa} and \ref{grsa}}

Throughout this section $A$ will denote a non-cyclic $p$-group of order $p^2$ for some prime $p$. As in the previous sections $H$ denotes the finite group of order $q$ that acts on $A$ in such a manner that $C_{A}(h)=0$ for every $h\in H\backslash\{1\}$.

\begin{lemma}\label{free1} Let $n$ be a positive integer and $h_1,\dots,h_n$ not necessarily distinct elements of $H\setminus\{1\}$. Suppose that there exists a non-zero element $a\in A$ such that $na=a^{h_1}+\dots+a^{h_n}$. Then $p\leq p_0$, where $p_0=p_0(n,q)$ is a constant depending only on $n$ and $q$. 
\end{lemma}
\begin{proof} Assume that the lemma is false and there exist  infinitely many primes $p_{1},p_{2},\ldots$ such that for each $i$ the elementary abelian $p_i$-group $A_i$ of rank two admits an action by $H$ with the properties that 
$ C_{A_i}(h)=  0$ for all $h\in H\setminus\{1\}$ and there exists $0\neq a_i\in A_i$ for which $na_{i}=a_{i}^{h_1}+\dots+a_{i}^{h_n}$.
  
We view each group $A_i$ as a  linear space over the field with $p_i$ elements. Thus, we obtain a family of  linear representations of $H$.   An ultraproduct with respect to  some non-principal ultrafilter on $\Bbb N$  of this family of representations is a representation $\rho:H \to GL(V)$, where $V$ is a $2$-dimensional vector space over a field of characteristic zero (see for example \cite[Appendix, Theorem B.6]{hall}). Furthermore it easy to check that $C_V(h)=0$ for all $h\in H\setminus\{1\}$ and $V$ contains a non-zero vector $v$ such that $nv=v^{h_1}+\dots+v^{h_n}$.

 This contradicts Lemma \ref{free}. The proof is complete.
\end{proof}

Formally, the constant $p_0$ in the above lemma also depends on the choise of
$h_1,\dots,h_n$. However it is clear that the number of all possible choices of $n$ nontrivial elements in $H$ is bounded in terms of $n$ and $q$ only. Therefore there exists an ``absolute" constant which is independent of the choice of $h_1,\dots,h_n$. In the sequel $p_0$ will always mean that ``absolute" constant. The next corollary is immediate and so we omit the proof.

\begin{corollary}\label{ppp} If $p>p_0(c,q)$ and $0\neq b\in A$, then the sequence $\underbrace{(b,b,\dots,b)}_c$ is $H$-independent.
\end{corollary}

\begin{proof}[Proof of Theorem \ref{posa}] Set $P=p_0(c,q)+q^{c+1}+1$ and assume that $p\geq P$. By \cite[Theorem 3.9]{caeufe} it is sufficient to show that $L$ satisfies the $c$th selective nilpotency condition, that is, $$[L_{a_{1}},L_{a_{2}},\ldots,L_{a_{c+1}}]=0 \ \textrm{ whenever } (a_1,a_{2},\ldots,a_{c+1})\in F_{c+1}$$
and that there is a $(c,q)$-bounded number $m$ such that $$[L_a,\underbrace{L_b,\ldots,L_b}_{m}]=0 \ \textrm{ for all } a,b\in A.$$ The selective nilpotency of $L$ can be established precisely as in the proof of Theorem \ref{cef}. Hence, it is sufficient to show that $L$ verifies the other condition. We will show first that $$[L_a,\underbrace{L_0,\ldots,L_0}_{c}]=0 \ \textrm{ for all } a\in A.$$ Indeed, since $A$ is noncyclic, we can choose a proper subgroup $B< A$ such that $0,a\in B$. By the hypothesis $L_B$ is nilpotent of class at most $c$. Hence, $[L_a,\underbrace{L_0,\ldots,L_0}_{c}]=0$. 

Now choose $0\neq b\in A$. By Corollary \ref{ppp} the sequence $\underbrace{(b,b,\dots,b)}_c$ is $H$-independent. Put $$D(b)=\{a\in A\vert\ [L_a,\underbrace{L_b,\ldots,L_b}_{c}]\neq0\} \text{ \ \ and \ \ } M=M(\underbrace{b,\dots,b}_c).$$ We have $D(b)\subseteq M\cup \{0\}$ and so by Lemma \ref{LemmaM} $|D(b)|\leq q^{c+1}+1$. As in the proof of Theorem \ref{tri} put $a_i=a+ib$, where $i=0,1,\dots,P-1$. Since $p\geq P$, it follows that $a_i\neq a_j$ whenever $i\neq j$. Let $d$ be the least index for which $a_d\not\in D(b)$. It is clear that $d\leq q^{c+1}+2$.  By the choice of $d$ we have  $$[L_{a_d},\underbrace{L_b,\ldots,L_b}_{c}]=0.$$ Taking into account that  $$[L_a,\underbrace{L_b,\ldots,L_b}_{d}]\leq L_{a_d}$$ write $$[L_a,\underbrace{L_b,\ldots,L_b}_{d+c}]\leq [L_{a_d},\underbrace{L_b,\ldots,L_b}_{c}]=0.$$ The theorem follows.
\end{proof}

We will now require some facts on coprime actions of finite groups (see for example \cite[6.2.2, 6.2.4]{go}). 
\begin{lemma}\label{11} Let $A$ be a group of automorphisms
of the finite group $G$ with $(|A|,|G|)=1$.   
\begin{enumerate}
\item If $N$ is an $A$-invariant normal subgroup
of $G$ we have $C_{G/N}(A)=C_G(A)N/N$;
\item  If $H$ is an $A$-invariant $p$-subgroup of $G$,
then $H$ is contained in an $A$-invariant Sylow $p$-subgroup of $G$;
\item If $P$ is an $A$-invariant Sylow $p$-subgroup of $G$, then
$C_P(A)$ is a Sylow $p$-subgroup of $C_G(A)$.
\item If $A$ is a non-cyclic elementary abelian group we have $G=\langle C_G(a)\ \vert\  a\in A\setminus\{1\}\rangle$.
\end{enumerate}
\end{lemma}
 
The next lemma and its corollary can be extracted from arguments given in \cite{ward73}. For the reader's convenience we supply a proof.

\begin{lemma}\label{nuj} Let $p$ be a prime and $A$ an elementary group of order $p^2$ acting on a finite $p'$-group $G$ in such a way that $C_G(a)$ is nilpotent for every nontrivial $a\in A$. Let $r$ be a prime dividing $|C_G(A)|$ and $x$ an $r$-element in $C_G(A)$. Then the index $[G:C_G(x)]$ is an $r$-power.
\end{lemma}
\begin{proof} If $G$ is not an $r$-group choose a prime $s\neq r$ dividing $|G|$. By Lemma \ref{11} $G$ possesses an $A$-invariant Sylow $s$-subgroup $S$ and $S=\langle C_S(a)\ \vert\ a\in A\setminus\{1\}\rangle$. We observe that for every $1\neq a\in A$ both $x$ and $C_S(a)$ are contained in the nilpotent subgroup  $C_G(a)$. Since $r\neq s$, it follows that $x$ centralizes $C_S(a)$. Taking into account that $S=\langle C_S(a)\ \vert\  a\in A\setminus\{1\}\rangle$ we conclude that $x$ centralizes $S$. Hence, the index $[G:C_G(x)]$ is an $r$-power, as required. 
\end{proof}

\begin{corollary}\label{solu} Let $p$ be a prime and $A$ an elementary group of order $p^2$ acting on a finite $p'$-group $G$ in such a way that $C_G(a)$ is nilpotent for every nontrivial $a\in A$. Then $G$ is soluble.
\end{corollary}
\begin{proof} The Burnside-Kazarin theorem \cite{kazarin} tells us that if $x$ is an element of $G$ such that the index $[G:C_G(x)]$ is a prime-power, then $x$ belongs to the soluble radical of $G$. In view of the above lemma we now conclude that $C_G(A)$ is contained in the soluble radical of $G$. If $C_G(A)=1$, then by Martineau's theorem \cite{martineau} $G$ is soluble. Putting these observations together with Lemma \ref{11} (1) we conclude that $G$ is soluble.
\end{proof}

Ward proved that under the hypothesis of Corollary \ref{solu} $G$ is actually metanilpotent \cite{ward} and examples show that $G$ need not be nilpotent.  Indeed, let $r$ and $s$ be odd primes and $H$ and $K$ the groups of order $r$ and $s$ respectively. Denote by $V=\langle v_1,v_2\rangle$ the four-group and by $Y$ the semidirect product of $K$ by $V$ such that $v_1$ acts on $K$ trivially and $v_2$ takes every element of $K$ to its inverse. Let $B$ be the base group of the wreath product $H\wr Y$ and note that $[B,v_1]$ is normal in $H\wr Y$. Set $G=[B,v_1]K$. The group $G$ is naturally acted on by $V$ and $C_G(v)$ is abelian for every nontrivial $v\in V$. It is clear that $G$ is not nilpotent when $r\neq s$.
Our next result shows however that under the hypothesis of Theorem \ref{grsa} $G$ is nilpotent. Actually the nilpotency of $G$ will be proved without any restrictions on the prime $p$.

\begin{proposition}\label{nilpo} If a finite group $G$ admits a coprime action by a Frobenius group $FH$ where $F$ is non-cyclic abelian and $C_{G}(H)$ and $C_{G}(a)$ are nilpotent for all $a\in F\setminus \{1\}$, then $G$ is nilpotent.
\end{proposition}
\begin{proof} We note first that if the rank of $F$ is greater than two, then nilpotency of $G$ follows from \cite{ward1}. A related result bounding the nilpotency class of $G$ was obtained in \cite{shu2001}.  Assume that $F$ is of rank two and the theorem is false. Let $G$ be a counter-example of minimal order. By Corollary \ref{solu} $G$ is soluble and so it easily follows that $G=RS$, where $R$ is a unique minimal normal elementary abelian $r$-subgroup for a prime $r$ and $S$ is an $F$-invariant Sylow $s$-subgroup of $G$ for some prime $s\neq r$. Suppose that $C_R(F)=C_S(F)= 1$. Then $C_G(F)=1$ and Theorem 2.7(c) in \cite{khumashu} tells us that $G$ is nilpotent, a contradiction. Hence, either $C_R(F)\neq 1$ or $C_S(F)\neq 1$.

Suppose $C_R(F)\neq 1$ and choose $1\neq x\in C_R(F)$. By Lemma \ref{nuj} the $[G:C_G(x)]$ is an $r$-power. Since $R$ is abelian, it follows that $R\leq C_G(x)$ and so $x\in Z(G)$. This leads to a contradiction since $R$ is a minimal normal subgroup of $G$.

Now suppose $C_S(F)\neq 1$ and choose $1\neq y\in C_S(F)$. By Lemma \ref{nuj} the $[G:C_G(y)]$ is an $s$-power. It follows that $y$ centralizes $R$. Thus, $C_S(R)\neq 1$ and this again leads to a contradiction since $R$ is a unique minimal normal subgroup of $G$.
\end{proof}

\begin{proof}[Proof of Theorem \ref{grsa}]
By Proposition \ref{nilpo} $G$ is nilpotent. We wish to show that $G$ is nilpotent of $(c,q)$-bounded class. Consider the associated Lie algebra of the group $G$
\begin{equation*}
L(G)=\displaystyle\bigoplus^{n}_{i=1}\gamma_i/ \gamma_{i+1},
\end{equation*}
where $n$ is the nilpotency class of $G$ and $\gamma_i$ are the terms 
of the lower central series of $G$ (see \cite{khu} for general information on $L(G)$ and how it relates with the structure of $G$). The action of the group $FH$ on $G$ induces naturally an action of $FH$ on $L(G)$. We extend the ground ring by a primitive $p$-th root of unity $\omega$ by setting $L=L(G)\otimes_{\mathbb{Z}}\mathbb{Z}[\omega]$. The action of $FH$ on $L(G)$ extends naturally to $L$.

Let $A$ be the character group of $F$, for which we use the additive notation. The group $A$ is non-cyclic of order $p^2$. We define an action of the group $H$ on $A$ by the 
rule $a^h(f)=a(f^h)$ and under this action we have $C_{A}(h)=0$ for each $h\in H\setminus\{1\}$.

For any $a\in A$ we set
\begin{equation*}
L_{a}=\{x\in L \ | \ x^f= a(f)x, \ \textrm{for each} \ f\in F\}.
\end{equation*}
Because the action is semisimple we have the ``generalized eigenspace decomposition'' 
$$L=\bigoplus_{a\in A}L_{a}$$ satisfying
\begin{equation*}
[L_a,L_b]\subseteq L_{a+b} 
\ \textrm{ for all} \ a,b\in  A.
\end{equation*}

Thus, $L$ becomes an $A$-graded Lie algebra. The action of the 
group $H$ on $L$ permutes the components $L_{a}$ in the 
following way: if $h\in H$, then $L_{a}^h=L_{a^{h^{-1}}}$, 
indeed, for $l\in L_{a}$ we have $$(l^{h})^{f}=(l^{hfh^{-1}})^{h}=(l^{f^{h^{-1}}})^{h}=
 a(f^{h^{-1}})l^{h}=a^{h^{-1}}(f)l^{h}$$
for any $f\in F$, so $l^{h}\in L_{a^{h^{-1}}}$.

 Because the action of $FH$ on $G$ is coprime, it is easy to see that the fixed-point subrings $C_{L}(H)$ is nilpotent of class at most $c$. Let $B$ be a proper subgroup of $A$. Since $B$ is cyclic, it follows that $F$ induces a cyclic group of automorphisms on $L_B$. Therefore $L_B$ is contained in $C_L(f)$ for some nontrivial $f\in F$. Thus, by the hypothesis, $L_B$ is nilpotent of class at most $c$. Now Theorem \ref{posa} tells us that $L$ is nilpotent of $(c,q)$-bounded class. The same is true for $L(G)$ and therefore for $G$. The proof is complete.
\end{proof}

\end{document}